\documentclass[twoside,11pt]{amsart}

\usepackage{amsmath,latexsym,amssymb,mathptm, times,verbatim, enumerate}

\input amssym.def
\input amssym
\input xypic
\input xy
\xyoption{all}
\def\A{\overline{A}}

\def\q{\mathfrak q}
\def\p{\mathfrak p}

\def\Proj{\operatorname {Proj}}
\def\Q{\mathfrak Q}
\def\P{\mathfrak P}
\def\Quot{\operatorname{Quot}}
\def\m{\mathfrak m}
\def\R{\mathcal R}
\def\G{\mathcal G}
\def\rank{\operatorname{rank}}
\def\fakeht{\vphantom{E^{E_E}_{E_E}}}
\def\core{\operatorname{core}}
\def\End{\operatorname{End}}

\newtheorem{theorem}{Theorem}[section]

\newtheorem{corollary}[theorem]{Corollary}
\newtheorem{proposition}[theorem]{Proposition}
\newtheorem{claim}[theorem]{Claim}
\newtheorem{observation}[theorem]{Observation}

\newtheorem*{Theorem}{Theorem}

\theoremstyle{definition}
\newtheorem{definition}[theorem]{Definition}
\newtheorem{remark}[theorem]{Remark}
\newtheorem*{Remark}{Remark}
\newtheorem{remarks}[theorem]{Remarks}
\newtheorem{data}[theorem]{Data}
\newtheorem{notation}[theorem]{Notation}

\newtheorem{chunk}[theorem]{}
\newtheorem*{proof1}{Proof of Theorem \ref{Thm.2}}
\newtheorem*{proof2}{Proof of Example \ref{Ex8}}
\numberwithin{equation}{theorem}

\DeclareMathOperator*{\cart}{\times}

\begin{document}

\baselineskip=16pt

\title[Blowups and fibers of morphisms]{\bf Blowups and fibers of morphisms}
\date\today

\author[Andrew R. Kustin, Claudia Polini, and Bernd Ulrich]
{Andrew R. Kustin, Claudia Polini, and Bernd Ulrich}

\thanks{AMS 2010 {\em Mathematics Subject Classification}.
Primary 13A30, 13H15, 14A10; Secondary 13D02,  14E05, 14F18.}

\thanks{The first author was partially supported by the the Simons Foundation.
The second author was partially supported by the NSF.
The last author was partially supported by the NSF and as a Simons Fellow.}



\address{Department of Mathematics, University of South Carolina,
Columbia, SC 29208} \email{kustin@math.sc.edu}

\address{Department of Mathematics, 
University of Notre Dame
Notre Dame, IN 46556} \email{cpolini@nd.edu}

\address{Department of Mathematics,
Purdue University,
West Lafayette, IN 47907}\email{ulrich@math.purdue.edu}

\begin{abstract} Our object of study is a rational map $\Psi:\xymatrix{\mathbb P_k^{s-1}\ar@{-->}[r]&\mathbb P_k^{n-1}}$ defined by homogeneous forms $g_1,\dots,g_n$, of the same degree $d$, in  the homogeneous coordinate ring $R=k[x_1,\dots,x_s]$ of $\mathbb P_k^{s-1}$.
Our goal is to relate
properties of $\Psi$,  of the homogeneous coordinate ring $A=k[g_1,\dots,g_n]$ of the variety parametrized by $\Psi$, and of the Rees algebra $\R(I)$, the bihomogeneous coordinate ring  of the graph of $\Psi$.
For a regular map $\Psi$, for instance, we prove that $\R(I)$ satisfies Serre's  condition $R_i$, for some $i>0$,  if and only if $A$ satisfies $R_{i-1}$ and $\Psi$ is birational onto its image.
Thus, in particular, $\Psi$ is birational onto its image if and only if
$\R(I)$ satisfies $R_1.$ Either condition has implications for the shape of the {\it core},  namely,
 ${\rm core}(I)$ is the multiplier ideal of $I^s$ and ${\rm core}(I)=(x_1,\ldots,x_s)^{sd-s+1}.$ Conversely, for $s=2$, either equality for the core implies birationality.
In addition, by means of the generalized rows of the syzygy matrix of $g_1,\dots,g_n$, we give an explicit method to reduce the
non-birational case to the birational one when $s=2$.
\end{abstract}

\maketitle

\section{Introduction.}

Let $k$ be a field. We investigate rational maps $\Psi:\xymatrix{\mathbb P_k^{s-1}\ar@{-->}[r]&\mathbb P_k^{n-1}.}$ Such a map $\Psi$ is defined by homogeneous forms $g_1,\dots,g_n$, of the same degree $d$, in  the homogeneous coordinate ring $R=k[x_1,\dots,x_s]$ of $\mathbb P_k^{s-1}$. 
One of our main goals is to determine necessary and sufficient conditions for $\Psi$ to be birational onto its image. 
This problem has been studied extensively; see, for example, Hacon \cite{H03} for a technique that uses the non-vanishing theorem of Koll\'ar \cite{K93}. 
The advantage of having convenient criteria for determining if a given parameterization is birational onto its image is quite obvious. Papers about parameterization of curves and surfaces are now ubiquitous (see, for example \cite{BD,BD13,BBC,bu,BB,BD'A,C06,CWL,CHW,CKPU,HSV,HSV12,BSSA,SCG}). It is much better to have ``the parameterization is birational'' as a conclusion, rather than as a hypothesis. 

In this paper we employ the syzygies of the forms 
$[g_1,\dots,g_n]$ to determine if the rational map $\Psi=[g_1:\dots:g_n]: \xymatrix{\mathbb P^{s-1}\ar@{-->}[r]& \mathbb P^{n-1}}$
is birational. For $s=n$ this approach appears  already in the work of Hulek, Katz, and Schreyer \cite{HKS}, 
and for $n\ge s$ it has been further developed in \cite{RS}.
In \cite{Si, DHS}  the method has been advanced 
by emphasizing
the role of the Rees algebra associated to the ideal $I=(g_1,\dots,g_n)$ of $R$. The Rees algebra $\R(I)$ gives the bi-homogeneous coordinate ring of the graph of $\Psi$, whereas, the subalgebra  $A=k[g_1,\dots,g_n]$ of $R$ is the homogeneous coordinate ring of the image of $\Psi$. In fact, $A$  is isomorphic to the special fiber ring $\mathcal F(I)=\R(I)/\m\R(I)$ for $\m$ equal to the maximal homogeneous ideal $\m=(x_1,\dots,x_s)$ of $R$. The rings $\R(I)$ and $\mathcal F(I)$ are known as {\it blow-up rings} associated to $I$. In this paper we relate geometric properties of $\Psi$, algebraic information about the homogeneous coordinate ring 
$A$ of the image, and the bi-homogeneous coordinate ring $\R(I)$  of the graph.

If $\Psi$ is a morphism (that is, if $I$ is primary to the maximal homogeneous ideal $\m$ of $R$), then the degree of $\Psi:\mathbb P_k^{s-1}\to \operatorname{Im}\Psi$ is equal to $d^{s-1}/e(A)$, where $e(A)$ is the multiplicity of the standard graded $k$-algebra $A$. We prove, for instance, that 
$\R(I)$ satisfies Serre's  condition $R_i$, for some $i>0$, if and only if $A$ satisfies $R_{i-1}$ and $\Psi$ is birational onto its image (that is, $\Psi:\mathbb P_k^{s-1}\to \operatorname{Im}\Psi$ has degree $1$). Thus, in particular, $\Psi$ is birational onto its image if and only if $\R(I)$ satisfies $R_1. $ Furthermore, $\Psi:\mathbb P_k^{s-1}\to \operatorname{Im}\Psi$ is a birational morphism with a  smooth  image if and only if the Rees ring $\R(I)$ has an isolated singularity. Even if the rational map $\Psi$ is not a morphism, if the dimension of $\operatorname{Im}\Psi$ is $s-1$ (that is, if the Krull dimension of $A$ is $s$), then the degree of $\Psi:\mathbb P_k^{s-1}\to \operatorname{Im}\Psi$  is equal to the multiplicity of the local ring $\R(I)_{\m\R(I)}$.
Moreover, we are able to relate the degree of the map $\Psi,$ the degree of the image, and the $j$-multiplicity of the ideal $I$. Results of this type were obtained before by Simis, Ulrich and Vasconcelos \cite{SUV}, Validashti \cite{Va}, Xie \cite{Xie} and by Jeffries, Montano, and Varbaro \cite{JMV}.  

The $j$-multiplicity is a generalization of the classical Hilbert-Samuel multiplicity that applies to ideals that are not necessarily zero-dimensional. The notion was introduced by
Achilles and Manaresi \cite{AM} and has found applications in intersection theory and equisingularity theory. It is interesting to find formulas for the $j$-multiplicity of classes 
of ideals \cite{NU, JM, JMV}. Our formula serves this purpose if the degree of the map and of its image are known. Conversely, we obtain the
degree of the image of  a rational map if the $j$-multiplicity can be computed, for instance by using residual intersection techniques (see, e.g. \cite{KPUGor}). We also express the degree of certain dual varieties 
in terms of the  $j$-multiplicity of Jacobian ideals.

The starting point for our investigation is the Eisenbud-Ulrich interpretation \cite{EU} of the fibers of the morphism $\Psi$ over the point $p$ in $\mathbb P^{n-1}$ in terms of the corresponding generalized row of the homogeneous syzygy matrix for $[g_1,\dots,g_n]$. This technique is explained and extended in Section~\ref{EU}; it also plays a significant role in \cite{CKPU}.

 In Section~\ref{Repar} we prove an algebraic analogue of a consequence of
Hurwitz' theorem. Let $r$ be the degree of the morphism 
$\Psi:\xymatrix{{\mathbb P}_k^{1}\ar@{->>}[r]&{\rm Im}(\Psi),}$
which at the level of coordinate
rings corresponds to the embedding $A
\hookrightarrow k[R_d] $. Thus $r$ is the degree of the field
extension ${\rm Quot}(A) \subset {\rm Quot}(
k[R_d] )$. 
 We 
show that there exist
homogeneous forms $f_1, f_2$ of degree $r$ in $R$ such that the entries
of the matrix $\varphi$ are homogeneous polynomials in the variables
$f_1$ and $f_2$. In particular, the ideal $I$ is extended from an ideal
in $k[f_1, f_2]$. Thus replacing $k[x_1,x_2]$ by $k[f_1, f_2]$ we can reduce the
non-birational case to the birational one.
Furthermore, we provide an explicit description of $f_1$ and $f_2$ in
terms of $\varphi$. 
If $q_1$ and $q_2$ are general points  in $\mathbb P_k^1$, then   $f_i=\operatorname{gcd}(I_1(p_i\varphi))$ with $p_i=\Psi(q_i)$.
This method gives an efficient algorithm for reparameterizing the rational map $\Psi$. 

In Section~\ref{Reesrings} we relate the birationality of $\Psi$ to the shape of  $\core(I)$, the core of $I$,  the
intersection of all reductions of $I$. Since reductions, even
minimal ones, are highly non-unique, one uses the core to encode
information about all of them.  The concept was introduced by Rees and Sally \cite{18}, and has
been studied further by Huneke and Swanson, by Corso, Polini, and Ulrich,  by Polini and Ulrich, and by Huneke and Trung  \cite{11,CPU1,CPU2,4,PU,HT}. The core appears naturally in the
contexts of Brian\c con-Skoda theorems that compare the integral
closure filtration with the adic filtration of an ideal \cite{BS, L}. Another aspect that makes the core very appealing is its connection 
 to adjoints and multiplier ideals, and, as
discovered by Hyry and Smith, to Kawamata's conjecture on the
non-vanishing of sections of certain line bundles \cite{14,14.5}.  We prove  that if $\Psi$ is birational
onto its image then ${\rm core}(I)$ is the adjoint ideal of $I^s$ and 
${\rm core}(I)=(x_1,\ldots,x_s)^{sd-s+1}.$ The converse of this statement holds for $s=2$. Indeed, 
if $\operatorname{Im}\Psi$ is a curve, $\Psi$ is a morphism, and $\varphi$ is a homogeneous Hilbert-Burch matrix for the row vector $[g_1,\dots,g_n]$, then in \ref{C3} we prove the following result.

\begin{Theorem} 
Statements {\rm(1)\,--\,(8)} are equivalent.
\begin{itemize}
\item[\rm(1)] The morphism  $\, \Psi$ is birational onto its image.
\item[\rm(2)] The Rees ring $\R(I)$ satisfies Serre's condition $(R_1)$. 
\item[\rm(3)] One has the equality of canonical modules $\omega_{\R(I)}= \omega_{\R(\m^d)}$.
\item[\rm(4)]  One has the equality of endomorphism rings  $\End_{\R(I)}(\omega_{\R(I)})=  \End_{\R(\m^d)}(\omega_{\R(\m^d)})$.
\item[\rm(5)] $e(A)=d$.
\item[\rm(6)] ${\rm core}(I)=\m^{2d-1}$.
\item[\rm(7)] One has the equality of the core and an adjoint ${\rm core}(I)={\rm
 adj}(I^{2})$.
 \item[\rm(8)] The ideal ${\rm core}(I)$ is integrally closed.
\end{itemize}
  Furthermore, statements {\rm(1)\,--\,(8)} are all  implied by
\begin{enumerate}
\item[{\rm (9)}]  $\gcd($column degrees of $\varphi)=1$.
\end{enumerate}
\end{Theorem}

We highlight the fact that the integral closedness of the core of $I$, which is a single graded component of the canonical module $\omega_{\R(I)}$, forces the shape of the entire canonical module. Also, we emphasize that these equivalent conditions may sometimes be read from numerical information about a homogeneous presentation matrix $\varphi$ for $I$ (that is, from information about the graded Betti numbers of the homogeneous ideal $I$ in the ring $R=k[x_1,x_2]$). In general, the sufficient condition (9) is far from necessary; however, if $I$ is generated by monomials, then condition (9) is equivalent to conditions (1)~--~(8).

The fact that the core can detect geometric properties was already apparent in the work of Hyry and Smith \cite{14,14.5} and in \cite{FPU}, where the Cayley-Bacharach property of $0$-dimensional schemes is  characterized  in
terms of the structure of the core of the maximal ideal of their homogeneous coordinate ring. The equality ${\rm core}(I)={\rm adj}(I^g)$ (where $g$ is the height of the ideal $I$)  has also been investigated by Hyry and Smith \cite{14,14.5} in their work on the conjecture of Kawamata. Adjoints of ideals in regular domains were introduced by Lipman \cite{L}, and in rings essentially of finite type over a field of characteristic zero they coincide with multiplier ideals, which play an important role in algebraic geometry due to their connection with vanishing theorems \cite{La}. For the core to be the adjoint of an ideal, it needs to be
integrally closed, which is
always the case if $\R(I)$ satisfies Serre's condition $R_1$ (see
\cite{PU}). Surprisingly, however,  in \ref{C3} the integral closeness of the core is sufficient to guarantee the equality between the core and the multiplier ideal.


\bigskip

\section{Notation, conventions, and preliminary results.}\label{Prelims}

\begin{chunk}\label{saturate}If $I$ and $J$ are ideals of a ring $R$, then the {\it saturation of $I$ with respect to $J$} is $I: J^{\infty}=\bigcup\limits_{i=1}^{\infty}(I: J^i)$. Recall that $I: J^{\infty}$ is obtained from $I$ by removing all primary components whose radical contain $J$. We write $\operatorname{gcd}$ to mean {\it greatest common divisor}. If $I$ is a homogeneous ideal in $k[x,y]$, then 
we denote the $\operatorname{gcd}$ of a generating set of $I$ by $\operatorname{gcd} (I)$; notice that this polynomial generates the saturation   $I: (x,y)^{\infty}$.
\end{chunk}

\begin{chunk}\label{Quot}If $R$ is a ring, then we write $\operatorname{Quot}(R)$ for the {\it total ring of quotients} of $R$; that is, $$\operatorname{Quot}(R)=U^{-1}R,$$ where $U$ is the set of non-zerodivisors on $R$. 
If $R$ is a domain, then the total  ring of quotients of $R$ is usually called the {\it quotient field} of $R$.
An extension of domains $A\subset B$ is called {\it birational} if $A$ and $B$ have the same quotient field. 

If $A\subset B$ is an extension of domains so that $\Quot(B)$ is algebraic over $\Quot(A)$, then 
we use the three notations $$[B:A], \quad\rank_AB, \quad \text{and}\quad[\Quot (B):\Quot (A)]$$ interchangeably. Indeed, in this situation,
\begin{equation}\label{quot}\Quot (A)\otimes_AB=\Quot(B).\end{equation}
The rank of $B$ as an $A$-module, denoted $\rank_AB$, is defined to be 
 the dimension of the left hand side  of (\ref{quot}) as a vector space over $\Quot(A)$. 
The dimension of the right side of (\ref{quot}), as a vector space over $\Quot(A)$,  is denoted $[\Quot (B):\Quot (A)]$.  
\end{chunk}

 \begin{chunk}\label{mult}
If $R$ is a Noetherian ring, $M$ is a finitely generated $R$-module of Krull dimension $s$, and $\mathfrak a$ is an ideal of $R$ with the Krull dimension of $R/\mathfrak a$ equal to zero, then {\it the multiplicity of the $R$-module $M$ with respect to the ideal $\mathfrak a$} is
$$e_{\mathfrak a}(M)=s!\lim_{n\to\infty}\frac{\lambda_R(M/\mathfrak a^n M)}{n^s},$$where $\lambda_R(\underline{\phantom{x}})$ represents the length of an $R$-module. If $R$ is local with maximal ideal $\mathfrak m$, then we often write $e(M)$ or $e_R(M)$ in place of $e_{\mathfrak m}(M)$. Similarly, if $R$ is a standard graded algebra over a field (see \ref{sga}) with maximal homogeneous ideal $\mathfrak m$, and $M$ is a graded $R$-module, then we often write $e(M)$ or $e_R(M)$ in place of  $e_{\mathfrak m}(M)$. 
In either case, we call the common value $e(M)=e_R(M)=e_{\mathfrak m}(M)$ the {\it multiplicity of the $R$-module $M$}; in particular, $e(R)$ means $e_R(R)$. 
\end{chunk}

 \begin{chunk}\label{sga} 
Let $R=\bigoplus_{0\le i}R_i$ be a  Noetherian graded algebra and $M$ be a finitely generated graded $R$-module. It follows that 
$$\lambda_{R_0}(M_n)= \lambda_R\left(\frac{\bigoplus_{n\le i}M_i}{\bigoplus_{n< i}M_i}\right).$$  This non-negative integer  is   the value of the  {\it Hilbert function}  of the $R$-module $M$ at $n$, denoted $\operatorname{HF}_M(n)$.

  Let $k$ be a field. The  graded algebra $R=\bigoplus_{0\le i}R_i$ is a {\it standard graded $k$-algebra} if $R_0$ is equal to $k$, $R_1$ is a finitely generated $R_0$-module, and $R$  is generated by $R_1$ as an algebra over $R_0$.  If $R$ is a standard graded $k$-algebra and $M$ is a finitely generated,  graded $R$-module with positive Krull dimension $s$, then  the multiplicity of $M$ may be expressed in terms of Hilbert functions: 
 $$e_R(M)=(s-1)!\lim\limits_{n\to \infty}\frac {\operatorname{HF}_M(n)}{n^{s-1}}.$$
\end{chunk}

\begin{chunk}\label{sg-lo}If $R$ is a standard graded $k$-algebra with maximal homogeneous ideal $\mathfrak m$, then
the multiplicity $e(R)$ of the standard graded $k$-algebra $R$ is equal to the multiplicity $e(R_{\mathfrak m})$ of the local ring $R_{\mathfrak m}$. 
\end{chunk}

\begin{chunk}\label{nagata} One important application of multiplicity is the following well-known theorem of Nagata \cite[40.6]{N}, see also \cite[Exer.~11.8, Exam.~11.1.11]{SH}. If $C$ is a Noetherian formally equidimensional local ring, then $e(C)=1$ if and only if $C$ is a regular local ring. When we apply this result, $C$ is a one-dimensional Noetherian local domain and the hypothesis that $C$ be formally equidimensional is auotmatically satisfied.
\end{chunk}
 
\begin{chunk}\label{assoc} We use the associativity formula for multiplicity, see for example, \cite[Cor.~4.6.8]{BH} or \cite[Thm.~11.2.4]{SH}. Let $R$  be a Noetherian  ring, $\mathfrak a$ be an ideal of $R$ with the Krull dimension of  $R/\mathfrak a$ equal to zero, and $M$ be a finitely generated $R$-module. Then
$$e_{\frak a}(M)=\sum_P \lambda_{R_P}(M_P)e_{\mathfrak a}(R/P),$$ where $P$ varies over the prime ideals in the support of $M$ with the property that the Krull dimension of  $R/P$ is equal to the Krull dimension of $M$. The proof makes use of a filtration of $M$ whose factors are cyclic modules defined by prime ideals of $R$.  \end{chunk}

The following result is elementary. See \cite[Prop.~6.1]{SUV} for a more sophisticated version and \cite[Cor.~4.6.9]{BH} for a local version. 
\begin{observation}\label{elem}  Assume $A\subseteq B$ is a module-finite extension of standard graded $k$-algebras which also are   domains. Then $e_B(B)=e_A(A)\operatorname{rank}_AB$.
\end{observation}
\begin{proof}One has $e_B(B)=e_A(B)=e_A(A)\rank_AB$. The first equality holds because $B$ has the same Hilbert function independent of whether $B$ is viewed as an $A$-module or a $B$-module. The second equality is the associativity formula for multiplicity. 
\end{proof}

\begin{chunk}\label{((0))}A dominant rational map $\Psi:\xymatrix{X\ar@{-->}[r]&Y}$ of projective varieties is {\it birational} if the induced map of function fields $K(Y)\hookrightarrow K(X)$ is an isomorphism. In general, the {\it degree} of the rational map $\Psi$ is the dimension of the field extension $[K(X):K(\operatorname{Im}\Psi)]$.\end{chunk}

\begin{remark}\label{degree} Let $\psi: A\to B$ be a 
homomorphism of standard-graded $k$-algebras, where $k$ is a field and $A$ and $B$ are domains. Assume that  $\psi(A_i)\subset B_i$ for all $i$, and that $\psi(A_+)\neq 0$. Then the degree of the rational map $\Proj (\psi):\Proj (B)\to \Proj (A)$ is $[\Quot (B):\Quot (\psi(A))]$. \end{remark}

\begin{proof}According to \ref{((0))}, the degree of $\Proj (\psi)$ is defined to be $[B_{(0)}:(\psi (A))_{(0)}]$, where $B_{(0)}$ and $(\psi (A))_{(0)}$ are the sub-fields of $\Quot (B)$ and $\Quot (\psi(A))$, respectively,  which consist of the homogeneous elements of degree zero. Let $a_1$ be an element of $A_1$ with $\psi(a_1)\neq 0$. It is easy to see that $\psi(a_1)$ is transcendental over $B_{(0)}$, $\Quot (B)=B_{(0)}(\psi(a_1))$, and $\Quot (\psi(A))= (\psi (A))_{(0)}(\psi(a_1))$. It follows that 
${[B_{(0)}:(\psi (A))_{(0)}]}=[\Quot (B):\Quot (\psi(A))]$.
\end{proof}

\begin{chunk} If $R=\bigoplus_{0\le i}R_i$ is a graded ring and $s$ is a positive integer, then the $s^{\text{th}}$ Veronese ring of $R$ is equal to $R^{(s)}=\bigoplus_{0\le i}R_{is}$. One regrades the Veronese ring in order to have the component of $R^{(s)}$ in degree $i$ be $R_{is}$. The $s^{\text{th}}$ Veronese of a graded module $M=\bigoplus M_i$ is formed in a similar manner: $M^{(s)}=\bigoplus M_{is}$, with $M_{is}$ having degree $i$. \end{chunk} 

\begin{chunk}\label{typical} Recall that the {\it Rees
algebra \/} $\mathcal R(I)$ of an ideal $I$ in a commutative ring $R$
is the graded subalgebra $R[It]$ of the polynomial ring $R[t]$. If $R$ 
has a distinguished maximal ideal $\mathfrak m$ (that is, if $R$ is graded with maximal homogeneous ideal $\mathfrak m$ or if $R$ is local with maximal ideal $\mathfrak m$), then the {\it special fiber ring} of $I$ is $\mathcal R(I)\otimes_R R/\mathfrak m$.

In the typical situation in the present paper, the ring $R$ will be a standard graded polynomial ring over a field $k$, and $I$ will be a homogeneous ideal of $R$ generated by homogeneous forms of the same  degree $d$, and, after regrading, the $k$-subalgebra $k[I_d]$ of $R$ will be the coordinate ring of a projective variety. In this case, there is  a $k$-algebra isomorphism from the projective coordinate ring $k[I_d]$, to the special fiber ring $\mathcal F(I)$. Indeed,
%
%
%
$$\xymatrix{
     &\mathcal R(I)\ar@{->>}[dr]\\
k[I_d]\cong k[I_dt]\ar@{^{(}->}[ur]{\phantom{xx}}\ar[rr]^{\cong}&& \mathcal R(I)\otimes_RR/\mathfrak m=\mathcal F(I),}$$where $\mathfrak m$ is the maximal homogeneous ideal of $R$.
 \end{chunk}

\begin{chunk}\label{reduction}Let $J\subset I$ be ideals in a commutative  Noetherian ring $R$. The following conditions are equivalent:
\begin{enumerate}[\rm(a)]
\item there exists a non-negative  integer $m$ with $JI^m=I^{m+1}$, and
\item the  Rees algebra $\mathcal R(I)$ is finitely generated as a module over $\mathcal R(J)$.
\end{enumerate}
When these conditions occur, one says that $J$ is   a {\it reduction}
 of $I$ or $I$ is {\it integral} over $J$. (For details, see, for example, \cite[8.21,~1.25,~1.1.1,~1.2.1]{SH}.)\end{chunk}

\begin{chunk}\label{2.14}
If $R$ is a Noetherian domain, then the ring  $S$ is an {\it$S_2$-ification of $R$} if  
$R\subset S\subset \Quot({R})$, $S$ is module-finite over $R$,  $S$ satisfies Serre's condition $(S_2)$ as an $R$-module, and for each $s\in S$, the ideal $R:_Rs$ of $R$ has height at least $2$. 
It follows from  \cite[2.3]{HH} that the $S_2$-fication of $R$ is unique and  from \cite[2.7]{HH} that if $R$ has a canonical module $\omega_R$ then 
$\operatorname{End}_R(\omega_R)$ is the $S_2$-ification of $R$.

\end{chunk}

\begin{chunk}\label{gri} The concept of  ``generalized row ideals''  appears widely in the literature; see, for example, \cite{G,EHU,EU}.
  Let $M$ be a matrix with entries in a $k$-algebra and $p$  be a non-zero row vector with entries from $k$, where $k$ is a field. A {\it generalized row} of $M$ is the product $pM$.
If $pM$ is a generalized row of $M$, then the ideal $I_1(pM)$ is called a {\it generalized row ideal} of $R$.
\end{chunk}

\begin{chunk}
Let $X$ be a topological space. We say that a {\it general point of $X$} has a certain property if there exists a  dense, open subset $U$ of $X$ so that every point of $U$ has the property.
\end{chunk}

\section{Fibers, multiplicity, and row ideals}\label{EU}
\begin{data}\label{d1}
Let $k$ be an infinite field and $\xymatrix{\Psi: \mathbb P_k^{s-1} \ar@{-->}[r]&  \mathbb P_k^{n-1}}$ be a rational map defined by $k$-linearly independent homogeneous forms $g_1,\dots, g_n$ of degree $d$ in $R=k[x_1,\dots,x_s]$,  a standard graded polynomial ring  over $k$ in $s$ variables  with maximal homogeneous ideal $\m$, $I$ be the homogeneous  ideal $(g_1,\dots, g_n)$ in $R$, and  $\varphi$ be a homogeneous syzygy matrix of $[g_1,\dots, g_n]$. Let $S=k[T_1, \dots, T_n]$ be a standard graded polynomial ring  over $k$ in $n$ variables. The map $\Psi$ corresponds to the $k$-algebra homomorphism $\psi:S\to R$, which sends  $T_i$ to $g_i$.  Let $A$ be the image  $k[I_d]$ of this homomorphism and $B$ be the Veronese ring $k[R_d]$. After regrading, we view $A\subset B$ as standard graded $k$-algebras. Notice that $A$ is the homogeneous coordinate ring of the image of $\Psi$. Notice also that, according to Remark \ref{degree}, $[B:A]$ is the degree of the rational map $\Psi$.
\end{data}

\begin{observation}\label{der} Adopt the data of {\rm \ref{d1}}.  If $I$ is $\m$-primary then  $d^{s-1}=e(A)[B:A]$.\end{observation}
\begin{proof} The hypothesis that the ideal $I$ is $\m$-primary forces the ring extension $A \subset B$ to be module-finite; since $\mathfrak m^{dm}\subset I$ for some m, hence $B_m=A_1B_{m-1}$. Observation~\ref{elem} yields  $e(A)[B:A]=e(B)$. On the other hand from the  Hilbert function of $B$ one sees  $e(B)=d^{s-1}$. \end{proof}


As in \cite{EU} we define the fibers of the rational map $\Psi$ to be the following schemes.

\begin{definition}\label{defEU}
Adopt the data  of \ref{d1}. 
Let $p$ be a rational closed point in $\mathbb P_k^{n-1}$ and let $\P \in \Proj(S)$ be the homogeneous prime ideal corresponding to $p$. The {\it fiber of $\Psi$ over $p$}, (denoted $\Psi^{-1}(p)$), is the scheme $\ \Proj\, (R/(\P R:_RI^{\infty}))\, .$
\end{definition}

\begin{remarks}\label{remEU} 
(1) Definition \ref{defEU} gives the correct notion of fiber as a set. Indeed, since the ideal $I$ is the extension to $R$ of the homogenous maximal ideal of $S$, the prime ideals of $\ \Proj\, (R/(\P R:_RI^{\infty}))$ correspond to the primes of $\ \Proj\, (R)$ that contract to the prime $\P$ of $\ \Proj\, (S) .$ In particular,  the rational closed points of $\ \Proj\, (R/(\P R:_RI^{\infty}))$ correspond to the points in the domain of $\Psi$ that map to the point $p$ of $\mathbb P_k^{n-1}$.

\smallskip \noindent (2) 
If $p$ is the point $[\alpha_1:\dots:\alpha_n]$ of $\mathbb P_k^{n-1}$, then $\mathfrak P$ is the prime ideal 
\begin{equation}\label{P}\mathfrak P=I_2\left(\bmatrix \alpha_1&\dots&\alpha_n\\T_1&\dots&T_n\endbmatrix\right)\end{equation} of $S$, and $\mathfrak PR$ and $\mathfrak PB$ are the extensions of $\mathfrak P$ to the rings $R$ and $B$, respectively, under the ring homomorphisms:
$$\xymatrix{
&&R\\
S\ar^{\psi}[r]&A\ar^{\rm{incl}}[ur]\ar[dr]_{\rm{incl}}&\\
&&B,\ar[uu]_{\rm{incl}}}$$where $\rm incl$ is the inclusion map. So, in particular, the ideals $\mathfrak PR$ and $\mathfrak PB$ both are generated by the $2\times 2$ minors of
\begin{equation}\label{PR}I_2\left(\bmatrix\alpha_1&\dots&\alpha_n\\g_1&\dots&g_n\endbmatrix\right).\end{equation}
\end{remarks}

\bigskip

\bigskip

Remark~\ref{fiber} shows that, not surprisingly, the fiber of $\Psi$, as defined in Definition~\ref{defEU}, does not change when the rational map $\Psi$  is composed with a $t$-uple embedding. Furthermore, if the point is general, then the fiber of $\Psi$ does not change when $\Psi$ is composed with a birational map.

\begin{remark}\label{fiber} Adopt the data  of {\rm \ref{d1}} and let ${\mathfrak Q}$ be the homogeneous prime ideal in $R$ which corresponds to the rational point $q$ in $\mathbb P^{s-1}_k$.

\smallskip \noindent (1)   
If $q$ is any point  in the domain of the rational map $\Psi$, then the ideals $({\mathfrak Q} \cap A)R :_R \m^{\infty}$ and $({\mathfrak Q} \cap A^{(t)})R :_R \m^{\infty}$ of $R$ are equal for all positive integers $t$, where $A^{(t)}$ denotes the $t$-Veronese sub-ring of  $A$. 

\smallskip \noindent (2) Let $C$ be a standard graded $k$-algebra with $A \subset C \subset B$ and assume that $C$ is birational over $A$.   If $q$ is a general rational point  in $\mathbb P^{s-1}_k$, then the ideals $({\mathfrak Q} \cap A)R :_R \m^{\infty}$ and $({\mathfrak Q} \cap C)R :_R \m^{\infty}$ of $R$  are equal. 
\end{remark}

\begin{proof}
\smallskip \noindent (1) It suffices to show that  $({\mathfrak Q} \cap A)R$ and $({\mathfrak Q} \cap A^{(t)})R$ are equal locally at any homogeneous relevant prime ideal of $R$ that contains either ideal. Any such prime ideal contracts to  $\Q\cap A$ in $A$, and $\Q \cap A$ is a relevant prime ideal of $A$ because $q$ is in the domain of $\Psi$. Therefore, the two ideals ${\mathfrak Q} \cap A$ and 
$(\Q \cap A^{(t)})A$ of $A$ coincide locally at $\Q \cap A$. 

\smallskip \noindent (2) In a similar manner it suffices to show that the two ideals   $\Q \cap C$ and $(\Q \cap A)C$ are equal  locally at the prime ideal $\Q \cap A$ of $A$. To see this, notice that   $A_{\Q \cap A}=C_{\Q \cap A}$ because the extension $A \subset C$ is birational and the point $q$ is general.  
\end{proof}

In the next few results we impose the hypothesis that the Krull dimension of the ring $A$ of Data~{\rm\ref{d1}} is $s$. One could also say that the ideal $I$ has maximal analytic spread. This hypothesis  holds  when $I$ is $\mathfrak m$-primary. Indeed, in this case, $B$ is finitely generated as an $A$-module, as was observed in the proof of Observation~\ref{der}; hence, $A$ and $B$ have the same dimension. One advantage of the hypothesis ${\rm dim} \, A=s$  is that the rings $C \subset A \subset B$ all have the same Krull dimension for any Noether normalization $C$ of $A$ and this allows for  multiplicity calculation involving these rings. 

In the next proposition we show that for a general $k$-rational point $q$ in $\mathbb P_k^{s-1}$  the multiplicity of the fiber over $\Psi(q)$ coincides with the degree of the field extension $\Quot(A) \subset \Quot(B)$. Notice that the rational map $\Psi$ is defined at such a point $q$ and that  $\Psi(q)$ is general in the image of $\Psi$.

\begin{proposition}\label{lemmaEU} Adopt Data~{\rm\ref{d1}} and assume that the Krull dimension of $A$ is equal to $s$. Then 
the equation 
$$[B:A]=e(R/(\p R :_R I^{\infty}))$$
holds,  where $\p$ is the homogeneous prime ideal in $A$ of $\, \Psi(q)$ for  a general rational point $q$ in $\mathbb P_k^{s-1}$. 
\end{proposition}

\begin{proof}No harm is done if we assume that $k$ is algebraically closed because the hypotheses and conclusions remain unchanged under this change of base. 

The rational map $\xymatrix{\Proj(R) \ar@{-->}[r]& \Proj(A)}$ is defined at general points $q$ and their images $p$, which correspond to prime ideals ${\mathfrak p}$, are general in $\Proj(A)$. Since $A$ and $B$ have the same Krull dimension, the field extension ${\rm Quot}(A) \subset \Quot(B)$ is algebraic and therefore finite. 
Since $\p$ is general, the Generic Freeness Lemma implies that $B_{\p}$ is free as an $A_{\p}$-module. Therefore 
$B_{\p}$ is a finitely generated  $A_{\p}$-module
and
\[ [\Quot(B) : \Quot(A)] = \mu_{A_{\p}}(B_{\p})= \lambda_{A_{\p}}(B \otimes_A k(\p)).
\]  
The ring $B \otimes_A k(\p)$ is Artinian and therefore $B$ is equal to the direct product $\displaystyle\cart_{\q'} B_{\q'}/\p B_{\q'}$, where $\q'$ varies over all primes in $\Proj(B)$ with $\q'\cap A=\p$. Therefore
\[\lambda_{A_{\p}}(B \otimes_A k(\p))= \sum_{\q'} \lambda_{A_{\p}}(B_{\q'}/\p B_{\q'})\, .
\]

Since $B_{\p}$ is a finitely generated  $A_{\p}$-module,  the field extension ${\rm Quot}(A/\p) \subset \Quot(B/\q')$ is algebraic, and therefore $A/\p$ and $B/\q'$  have the same  the Krull dimension.  Thus the rings $A/\p$ and $B/\q'$ are standard graded one-dimensional domains over the algebraically closed field $k$, hence both are polynomial rings in one variable. Since the inclusion $A/\p \subset B/\q'$ is homogeneous, it then follows that it is actually an equality. Therefore  $k(\p) = k(\q')$  and we obtain 
\[ \lambda_{A_{\p}}(B_{\q'}/\p B_{\q'})=\lambda_{B_{\q'}}(B_{\q'}/\p B_{\q'})\, .
\]
Since $\Proj (B/\p B)\simeq \Proj(R/\p R)$, there exists a unique $\Q'\in \Proj(R)$ with $\Q' \cap A =\p$ corresponding to each $\q'$ and furthermore  \[\lambda_{B_{\q'}}(B_{\q'}/\p B_{\q'})=\lambda_{R_{\Q'}}(R_{\Q'}/\p R_{\Q'}) \, .\]

In summary we obtain
\[ [\Quot(B) : \Quot(A)]=\sum_{\q'} \lambda_{A_{\p}}(B_{\q'}/\p B_{\q'})=\sum_{\Q'} \lambda_{R_{\Q'}}(R_{\Q'}/\p R_{\Q'})\, .
\] 
Let $\m_A$ denote the homogeneous maximal ideal of $A$. Notice that $\m_AR=I$ and recall that ${\rm dim}\, A/\p$ is one. It follows that
\begin{align}\notag \{\Q' \mid \Q'\in \Proj(R), \, \Q' \cap A =\p \} &= \{\Q' \mid \Q'\in \Proj(R),\,  \p \subset \Q', \, \m_A \not\subset \Q'  \} \\
\notag &= \{\Q'  \mid  \Q'\in \Proj(R),\,  \p R \subset \Q', \, I \not\subset \Q'  \}\\
\notag &=\, \Proj(R) \cap V(\p R :_R I^{\infty})\, .
\end{align}

\noindent
The rings $R/\Q'$ are polynomial rings in one variable for every prime ideal $\Q'$ in the  above set. In particular, these prime ideals correspond to the minimal primes of maximal dimension of the ring $R/(\p R:_R I^{\infty})$. Now  the associativity formula for multiplicity implies that 
\[ \sum_{\Q'} \lambda_{R_{\Q'}}(R_{\Q'}/\p R_{\Q'}) =e(R/(\p R:_R I^{\infty}))\, ,\]
which completes the proof. 
\end{proof}

In the next corollary we make use of a Noether normalization $C$ of $A$ to compute the multiplicity of $A$ and express it in terms of the field degree $[A:B]$ and of the multiplicity of a ring defined by a colon ideal in $R$. 

\begin{corollary}\label{nov23} Let $k$ be an infinite field, $R=k[x_1,\dots,x_s]$ be a standard graded polynomial ring  over $k$ in $s$ variables, $I$ be a homogeneous  ideal  in $R$ generated by forms  of degree $d$, $A$ be the  subring $k[I_d]$, and $B$ be the Veronese subring $k[R_d]$. Assume that the Krull dimension of $A$ is equal to $s$ and let $f_1, \ldots, f_{s-1}$ be general $k$-linear combinations of homogenous minimal generators of $I$. Then
\[e(A)=\frac{1}{[B:A]} \cdot \, e\left(\frac{R}{(f_1, \ldots, f_{s-1}):_RI^{\infty}}\right)\, .
\]
\end{corollary}

\begin{proof} Let $f_1, \ldots, f_{s}$ be general $k$-linear combinations of homogenous minimal generators of $I$. Let $C=k[f_1, \ldots, f_s]$. Since $A$ has dimension $s$, the homogenous ring extension $C\subset A$ is module-finite  and $C$ is a polynomial ring. Hence $[A:C]=e(A)$ by Observation \ref{elem}, for instance.

We have the inclusion of domains
\[ C \subset A \subset B\, .\]
We compute the field degree $[B:C]$ in two different ways. First, 
\[ [B:C]=[B:A]\cdot e(A)\, .\]
\noindent
Next, we wish to apply Proposition \ref{lemmaEU} to express $[B:C]$. By the general choice of $f_1, \ldots, f_{s}\, $, we may assume that $[0:\ldots:0:1]$ is a general point of $\Proj (C) \simeq \Bbb P_k^{s-1}$.  The homogenous prime ideal of $C$ corresponding to this point is $\p=(f_1, \ldots, f_{s-1})C$. Therefore Proposition \ref{lemmaEU} shows that
\[[B:C]=e\left(\frac{R}{(f_1, \ldots, f_{s-1}):_RI^{\infty}}\right)\, ,\]
which yields the desired conclusion. 
\end{proof}

We promised in the introduction to relate the degree of the rational map $\xymatrix{\Psi: \mathbb P_k^{s-1} \ar@{-->}[r]&  \operatorname{im} \Psi}$ of Data~\ref{d1}, the degree of the image, and the $j$-multiplicity of the ideal $I$. The relation is 
$$j(I)=d[B:A]e(A).$$ A very quick proof of this formula is available in the present situation. Another proof of this formula, in a more general setting, is given in Theorem~\ref{rees}.

\begin{corollary}\label{nov24} Adopt the data of {\rm \ref{d1}}.  Then
$$j(I)=d[B:A]e(A).$$
\end{corollary}
\begin{proof}
According to \cite[2.5]{Xie}  (see also  \cite[3.8]{AM} and \cite[3.6]{NU})
$$j(I)=\lambda_R\left(\frac{R}{((f_1, \ldots, f_{s-1}):_RI^{\infty},f_s)}\right),$$ in the language of Corollary~\ref{nov23}. On the other hand, 
$f_s$ is regular on $\frac{R}{(f_1, \ldots, f_{s-1}):_RI^{\infty}}$, and $f_s$ is a homogeneous element of $R$ of degree $d$; so, 
$$\lambda_R\left(\frac{R}{((f_1, \ldots, f_{s-1}):_RI^{\infty},f_s)}\right)=d\cdot e\left(\frac{R}{(f_1, \ldots, f_{s-1}):_RI^{\infty}}\right),$$ and the assertion follows from Corollary~\ref{nov23}.
\end{proof}

 We now relate the fibers of the morphism $\Psi$ of Data~\ref{d1} to the generalized row ideals of the presentation matrix $\varphi$ of $I$. This material is taken from \cite{EU}. For a point $p \in \mathbb P_k^{n-1}$ the ideal $I_1(p\varphi)$ is well-defined: one may use any representative of  $p$ when computing the matrix product $p\varphi$.

\begin{observation}\label{obsEU} Adopt Data~{\rm\ref{d1}}. If $\P$ is the homogeneous prime ideal in $S$ which corresponds to the rational point $p$ in $\mathbb P^{n-1}_k$, then 
 the ideals $\P R:I^\infty$
and $I_1(p\varphi):I^\infty$
of $R$
are equal.
\end{observation}

\begin{proof}Let $p=[\alpha_1:\dots:\alpha_n]$ with $\alpha_i$ in $k$. Recall the generating sets given in (\ref{P}) and (\ref{PR}) for  the ideals $\P$ and $\P R$ of $S$ and $R$, respectively. Let $\chi$ be an invertible matrix with entries in $k$ and $(0,\dots,0,1)=(\alpha_1,\dots,\alpha_{n})\chi$. 
Define $[g_1',\dots,g_{n}']=\pmb g'$ by
\begin{equation}\label{.gi0} \pmb g' =[g_1,\dots,g_{n}]\chi.\end{equation}The entries of $\pmb g'$ generate $I$ and $\varphi'=\chi^{-1}\varphi$
is a homogeneous syzygy matrix   for $\pmb g'$. One consequence of this last statement is the fact that the bottom row of $\varphi'$ generates the ideal 
$(g_1',\dots,g_{n-1}'): g_{n}'$.
  On the other hand, 
$$\bmatrix \alpha_1&\dots&\alpha_{n}\\g_1&\dots&g_{n}\endbmatrix \chi= \bmatrix 0&\dots&0&1\\g_1'&\dots&g_{n-1}'&g_{n}'\endbmatrix;$$ and  the ideal of $2\times 2$ minors of a matrix is unchanged by row and column operations; therefore, $\mathfrak P R$ is generated by $(g_1',\dots,g_{n-1}')$. We now see that
\begin{equation}\label{BBB}I_1(p\varphi)=I_1([0,\dots,0,1]\varphi')=(g_1',\dots,g_{n-1}'): g_{n}'= \mathfrak P R: I;\end{equation} 
hence, $I_1(p\varphi):I^\infty=(\mathfrak P R: I):I^{\infty}=\mathfrak P R:I^{\infty}$. \end{proof}


Recall that the  degree of a subscheme of projective space  is the multiplicity of its homogenous coordinate ring. 

\begin{corollary}\label{corEU}
Adopt the data of {\rm \ref{d1}}. Let $p$ be a rational point in $\mathbb P^{n-1}_k$.
\begin{enumerate}[\quad\rm(1)]
\item If $k$ is algebraically closed, then $p$  is in the image of $\, \Psi$ if and only if  $I_1(p\varphi):I^\infty\neq R$.
\item If $p$ is in the image of $\, \Psi$, then the degree of the fiber of $\, \Psi$ over $p$ is equal to
the multiplicity of $R/(I_1(p\varphi):I^\infty)$.
\item If the Krull dimension of $A$ is equal to $s$ and $p=\Psi(q)$ for $q$  a general point in $\mathbb P_k^{s-1}$, then  $e(R/(I_1(p\varphi):I^\infty))=[\Quot(B):\Quot(A)]$.
\end{enumerate}
\end{corollary}

\begin{proof} 
For part (1) we use Observation~\ref{obsEU}, Remark~\ref{remEU}(1) and the fact  that the ideal  $I_1(p\varphi):I^\infty$  is not the unit ideal if and only it has dimension at least one. Item (2) follows from  Observation~\ref{obsEU}. To prove item (3) we apply Observation~\ref{obsEU} and  Proposition~\ref{lemmaEU}.
\end{proof}

More information  about the rational map $\Psi$ of  Data~\ref{d1} may be read from $\varphi$ when $I$ is an $\m$- primary ideal. In this case, $\Psi$ is actually a morphism and the image of $\Psi$ is a closed subscheme of $\mathbb P^{n-1}_k$.  Even more information can be read when $s=2$. In this case the saturation of the row ideal $I_1(p\varphi)$ is the ideal generated by  the $\operatorname{gcd}$ of the entries of the product matrix $p\varphi$  and the image of $\Psi$ is a curve $\mathcal C$.

\begin{corollary}\label{3.7} Adopt the data of {\rm \ref{d1}}. In addition assume that $I$ is $\m$-primary and $s=2$. Let $\mathcal C$ be the image of $\, \Psi$. The following statements hold.
\begin{enumerate}[\quad\rm(1)]
\item The point $p$ is on the curve $\mathcal C$ if and only if  $\operatorname{gcd}(I_1(p\varphi))$ is not in $k$.
\item If $p$ is on the curve $\mathcal C$, then the multiplicity of the fiber of $\, \Psi$ over $p$ is equal to
$$\deg (\operatorname{gcd}(I_1(p\varphi))).$$
\item If $p=\Psi(q)$ for $q$  a general point in $\mathbb P_k^1$, then  $\deg\operatorname(\operatorname{gcd}(I_1(p\varphi)))=[\Quot(B):\Quot(A)]$.
\item All entries of the matrix $\varphi$ have degree at least $[\Quot(B):\Quot(A)]$. 
\end{enumerate}\end{corollary}

\begin{proof} Items (1)-(3) follow from Corollary~\ref{corEU}. As for item (4),  notice that each column of $\varphi$ contains a non-zero element, hence the linear space 
\[V=\{ [\beta_1, \ldots, \beta_n] \in k^n \mid \sum \beta_i\varphi_{ij}=0\}
\] is a proper linear subspace of $\mathbb P_k^{n-1}$. Because $g_1,\dots,g_{n}$ are $k$-linearly independent  the image of $\Psi$ cannot be contained in $V$.  Thus  if $q$ is a general point in $\mathbb P_k^1$ every entry in the product $p\varphi=\Psi(q) \varphi$ is non-zero. Now item (4) follows from (3).
\end{proof}

\section{Reparameterization.}\label{Repar}

\begin{data}\label{data0} 
(1) Let $R=k[x_1, x_2]$ be a standard graded polynomial ring  over an  infinite perfect field $k$ with maximal homogeneous ideal $\m$,  
  $I$ be a height two  ideal of $R$ which is generated by homogeneous forms of the same degree $d$, $A$ and $B$ be the $k$-subalgebras $A=k[I_d]$ and $B=k[R_d]$ of $R$, and   $r$ be the degree of the field extension $\operatorname{Quot}(A)\subset \operatorname{Quot}(B)$. View $A$ and $B$ as standard graded $k$-algebras with maximal homogeneous ideals $\mathfrak m_A$ and $\mathfrak m_B$, respectively. Let $e=e(A)$.

\smallskip\noindent (2) Let $g_1,\dots,g_{n}$ be homogeneous forms in $R$ of degree $d$ which minimally generate the ideal $I$ and consider the morphism 
$\xymatrix{\Psi:\mathbb P^1_k\ar@{->}[r]& \mathbb P_k^{n-1}}$ which is defined by these forms; that is, $\Psi$ sends the  point $q$ of $\mathbb P_k^1$ to the point $[g_1(q):\cdots:g_{n}(q)]\in \mathbb P_k^{n-1}$. 
The corresponding homomorphism of $k$-algebras is
$\psi:S=k[T_1,\dots,T_{n}]\to R$, which sends $T_i$ to  $g_i$, for each $i$.

\smallskip\noindent (3) Let $\varphi$ be a minimal homogeneous Hilbert-Burch matrix for the row vector $[g_1,\dots,g_{n}]$. In other words, there exist shifts $D_j$ and a unit $u$ in $k$ so that 
\begin{equation}\label{HB}\xymatrix{0\to \bigoplus\limits_{j=1}^{n-1}R(d-D_j)\ar[r]^{\hskip15pt\varphi}&R(-d)^n\ar[rr]^{\ \ \ [g_1,\dots,g_{n}]}&&R}\end{equation} is 
a minimal homogeneous resolution of $R/I$ and 
$g_i$ is equal to $u(-1)^i$ times the determinant of $\varphi$ with row $i$ deleted, for $1\le i\le n$.  
\end{data}

In this section, we prove an algebraic analogue of a consequence of
Hurwitz' theorem. Let $r$ be the degree of the morphism 
$\Psi:\xymatrix{{\mathbb P}_k^{1}\ar@{->>}[r]&{\rm Im}(\Psi),}$
which at the level of coordinate
rings corresponds to the embedding $A
\hookrightarrow k[R_d] $. Thus $r$ is the degree of the field
extension ${\rm Quot}(A) \subset {\rm Quot}(
k[R_d] )$. 
 We 
show that there exist
homogeneous forms $f_1, f_2$ of degree $r$ in $R$ such that the entries
of the matrix $\varphi$ are homogeneous polynomials in the variables
$f_1$ and $f_2$. In particular, the ideal $I$ is extended from an ideal
in $k[f_1, f_2]$. Thus, replacing $k[x_1,x_2]$ by $k[f_1, f_2]$, we can reduce the
non-birational case to the birational one.
Furthermore, we provide an explicit description of $f_1$ and $f_2$ in
terms of $\varphi$. 
If $q_1$ and $q_2$ are general points  in $\mathbb P_k^1$, then   $f_i=\operatorname{gcd}(I_1(p_i\varphi))$ with $p_i=\Psi(q_i)$.

\medskip
The first half of the present section is devoted to the proof of the main result, Theorem~\ref{repar}. In the second half of the section, beginning with Corollary~\ref{H}, we use Theorem 4.2 to reduce the general (that is, non-birational) case to the birational one.

\begin{theorem}\label{repar} Adopt the data of {\rm\ref{data0}}. If $q_1$ and $q_2$ are general points  in $\mathbb P_k^1$, then $A\subset k[f_1, f_2]$, where  $f_i=\operatorname{gcd}(I_1(p_i\varphi))$ with $p_i=\Psi(q_i)$.
\end{theorem}

The proof of Theorem~\ref{repar} involves a number of steps. Before getting started we record some numerical information.
 In preparation for the first step in the proof of Theorem~\ref{repar}, we record a result from \cite[Cor.~5.2]{IEaI}.
Recall that a standard graded Cohen-Macaulay $k$-algebra $C$ over a
field $k$ with homogeneous maximal ideal $\m_C$  is said to have {\it minimal multiplicity} if its multiplicity
has the smallest possible value, namely, $e({C})=\mu(\m_C)-{\rm dim}({C})+1$. 
\begin{proposition}\label{6.25} Let $k$ be a perfect field and $A\subset B$ a homogeneous integral extension of standard graded $k$-domains. Further assume that $A$ is normal and Cohen-Macaulay. If $B$ has minimal multiplicity, then so does $A$. \end{proposition}

\begin{claim}\label{4.5} Adopt the data of {\rm \ref{data0}}. Theorem {\rm\ref{repar}} holds if the ring $A$ is normal.\end{claim}

\begin{proof} Apply Proposition \ref{6.25}. In the present situation, $A$ and $B$ are normal of dimension $2$  and therefore Cohen-Macaulay, and $B$ has minimal multiplicity since $d=e(B)=\mu(\m_B)-1$. It follows that $A$ has minimal multiplicity as well, namely,
\[e=\mu(\m_A)-1=\mu(I)-1=n-1\, .\]
We conclude that the syzygy matrix $\varphi$ has size $e+1 \times e$. Each entry in  column $j$ of $\varphi$ is a form of degree $D_j$, in the language of \ref{HB}, and, according to Corollary~\ref{3.7}(4), $D_j \ge r$. The $g_i$ have degree $d$ and each one of them is the determinant of an $e \times e$ minor of $\varphi$;  thus, $d=\sum_{j=1}^eD_j \ge er=d$, where the last equality follows from Observation~\ref{der}. We conclude that each entry of $\varphi$ is a homogeneous form of degree $r$.

Let $q_1, \ldots, q_{e+1}$ be general points in $\mathbb P^1_k$. Hence  the points $\Psi(q_1), \ldots,\Psi( q_{e+1})$ are general on the curve $\mathcal C =\operatorname{Im}\Psi$, and since $\mathcal C$ is non-degenerate in $\mathbb P^e$, it follows that  $\Psi(q_1), \ldots,\Psi( q_{e+1})$ span $\mathbb P^e$. In other words, the $(e+1) \times (e+1)$ scalar matrix $\Gamma$, whose rows are $\Psi(q_i)$, is invertible. 
We apply Corollary \ref{3.7}(3) to see that each $\operatorname{gcd}(I_1(\Psi(q_i)\varphi))$ is a form in $R$  of degree $r$. Write $f_i=\operatorname{gcd}(I_1(\Psi(q_i)\varphi))$. Degree considerations show that  
 \begin{equation}\label{lambda-i}\Psi(q_i)\varphi=f_i \cdot \lambda_i,\end{equation}
 where $\lambda_i$ is a vector of scalars. It follows that \begin{equation}\label{Gamma}\Gamma \varphi= D \Lambda,\end{equation} where $D$ is the diagonal matrix whose diagonal entries are the forms $f_1, \cdots, f_{e+1}$, and  $\Lambda$ is the ${(e+1)\times e}$ scalar matrix whose rows are the $\lambda_i$.  We notice that $\Lambda$ has maximal rank $e$ because $\varphi$ does. 

Let $q$ be another general point  in $\mathbb P^1$. Write \begin{equation}\label{again}\Psi(q)\varphi=f\lambda,\end{equation} as in \ref{lambda-i}, where $f$ is a form of degree $r$ in $R$ and $\lambda$ is a row vector of scalars from $k$. 
There is an $e\times (e-1)$ matrix of scalars $\Delta$ so that $\lambda \Delta=0$  and $\Delta$ has full rank $e-1$. Combine
\ref{again}, the definition of $\Psi$ as given in \ref{data0}(2), and \ref {Gamma} to see that
$$0=f \lambda \Delta=\Psi(q) \varphi \Delta =[g_1(q),\dots,g_{e+1}(q)]\Gamma^{-1}D\Lambda \Delta.$$ The point $q$ in $\mathbb P^1$ is general; so the product 
$[g_1(q),\dots,g_{e+1}(q)]\Gamma^{-1}$ is a row vector of non-zero scalars  and the product 
$[g_1(q),\dots,g_{e+1}(q)]\Gamma^{-1}D$ is equal to $[f_1,\dots,f_{e+1}]D'$, where $D'$ is an invertible diagonal matrix of scalars. Thus,
$$0=[f_1,\dots,f_{e+1}]D'\Lambda \Delta,$$ where $D'\Lambda \Delta$ is an $(e+1)\times (e-1)$ matrix of scalars of rank $e-1$. We conclude that the vector space $V$ spanned by the homogeneous forms $f_1,\dots,f_{e+1}$ has dimension at most two. Select a generating set $f_i,f_j$ for $V$. Every entry of the Hilbert Burch matrix $\varphi$ is an element of the ring $k[f_i,f_j]$ (as can be seen from \ref{Gamma}); hence $A=k[I_d]\subset k[f_i,f_j]$ and the proof of Claim ~\ref{4.5} is complete.
 \end{proof}

Next we are ready to prove Theorem ~\ref{repar}

\begin{proof} We prove the theorem by reducing to the normal case. The integral closure $\overline{A}$ of $A$ in $\Quot(A)$ is a graded $k$-subalgebra of $B$, though not necessarily standard graded. However, $\A$ is a finitely generated graded module over the standard graded $k$-algebra $A$ and  hence  there exists a positive integer $s$ such that  the $t$-Veronese sub-ring of the integral closure of $A$, denoted $\overline{A}^{(t)}$, is standard graded for all $t\ge s$.
 We observe that $\overline{A}^{(t)}=\overline{A^{(t)}}$.

 We apply Claim \ref{4.5} to $\overline{A^{(s)}}$. For  general points $q_1$ and $q_2$ in $\mathbb P^1_k$ we obtain  forms $\widetilde{f_1}$ and $\widetilde{f_2}$ such that  $\overline{A^{(s)}}\subset k[\widetilde{f_1}, \widetilde{f_2]}$. Let 
 $\mathfrak{Q_i}$ be the prime ideals in $R$ corresponding to $q_i$. We have 
\[
 \widetilde{f_i}R=(({\mathfrak{Q_i}} \cap \overline{A^{(s)}})R) : \m^{\infty}= (({\mathfrak{Q_i}} \cap A^{(s)})R):\m^{\infty} =((\mathfrak{Q_i} \cap A)R):\m^{\infty}=f_iR\, ,
\]
where the first and last equality follow from Observation~\ref{obsEU} (and \ref{saturate}) and the second and third equality hold due to  Remark ~\ref{fiber}. We conclude that $ \widetilde{f_i}k=f_ik$ and therefore $\overline{A^{(s)}}\subset k[f_1, f_2]$. In the same manner we obtain  $\overline{A^{(s+1)}}\subset k[f_1, f_2]$. Hence for every non-zero homogeneous element $\alpha$ in $A$ we have \[\alpha=\frac{\alpha^{s+1}}{\alpha^s}\in  {\rm Quot}(k[f_1, f_2])\, .\]
On the other hand, $\alpha$ is integral over $A^{(s)}$, hence integral over $k[f_1, f_2]$. We  conclude that $\alpha \in k[f_1,f_2]$  since $k[f_1,f_2]$ is normal. 
\end{proof}

\begin{corollary}\label{H} 
Adopt the data of {\rm\ref{data0}}. Let $f_1$ and $f_2$ as in Theorem~{\rm\ref{repar}}.
\begin{enumerate}[\quad\rm(1)]
\item The elements $f_1$, $f_2$ are a regular sequence of forms of degree $r$ in $R$.
\item The ring $R'=k[f_1,f_2]$ is a polynomial ring in the variables $f_1$ and $f_2$.  After regrading, we view $R'$ as a standard graded $k$-algebra.
\item  The vector space $I_d$ is a subset of  $ R'$ and the ideal $I'=I_dR'$ is  generated by forms of  degree $d/r$ in $R'$.
\item The extension $A=k[I'_{d/r}]\subset k[R'_{d/r}]$ is birational.
\item  The ring  $R$ is a  free $R'$-module of  rank $r^2$ and 
$I=I'R$.
\item There exists a minimal homogeneous Hilbert-Burch matrix for the row vector $[g_1,\dots,g_{n}]$  whose entries are in the ring $k[f_1,f_2]$.  \end{enumerate}
\end{corollary}
\begin{proof}
The forms $f_1$ and $f_2$ have degree $r$ according to Corollary~\ref{3.7}(3). From Theorem~\ref{repar} we obtain $I_d\subset R'=k[f_1,f_2]$. In particular,   $I\subset (f_1, f_2)R$ because $I=I_dR$. It follows that $f_1$, $f_2$ form a regular sequence  in $R$ and hence $R$ is a free module over $R'=k[f_1,f_2]$, necessarily  of rank $r^2$. Parts (1) and (5) are established. At this point, items (2),  (3), and (6) are clear. 

To prove (4) we use the following diagram, where   the relevant field degrees
are displayed:
\[\xymatrix{ &R\ar@{-}[dr]^{r^2}\\ 
&&R'\ar@{-}[d]_{d/r}\\
k[R_d]  \ar@{-}[uur]^{d}&&k[R'_{d/r}]\\
&A\ar@{-}[ul]_{r}\ar@{-}[ur] }\]

The diagram shows $\left[k[R'_{d/r}]:A\right]=1$.
\end{proof}

In the following corollary we gather important numerical consequences of Theorem~\ref{repar}.
\begin{corollary}\label{corn} Adopt the data of {\rm \ref{data0}}.   
\begin{enumerate}[\quad\rm(1)]
\item  $[B:A] \mid
{\rm deg} [\varphi]_{ij}$.
\item $[B:A] \mid {\rm gcd}(\mbox{column degrees of \,} \varphi)$.
\item   If {\rm gcd}$($column degrees of
$\varphi)=1$, then $A \subset B$ is a birational extension; in other words,
$\Psi$ is birational onto its image. 
\item  If  every entry of $\varphi$ has the same degree and that degree is a prime integer, then the morphism $\Psi$ is birational onto its image if and only if $\mu(I_1(\varphi))\ge 3$.
\end{enumerate}
\end{corollary}
\begin{proof} Items (1) and (2) follow from items (1) and (6) of Corollary~\ref{H}.  Item  (3) is now clear.  Item (4) is \cite[0.11]{CKPU}. 
\end{proof}

\begin{observation}\label{min-ideal-col-deg} Adopt the data of {\rm \ref{data0}} and assume that $I$ is a monomial ideal.  Then \[[B:A]={\rm gcd}(\mbox{column degrees of \,} \varphi)\, .\]
\end{observation}
\begin{proof}
We may assume that  $g_i=x^{a_i}y^{d-a_i}$ with $0=a_1 < a_2 <  \cdots < a_n=d$. Hence in the language of ~\ref{HB}, 
\[D_j=a_{j+1}-a_j \qquad \mbox{for} \qquad 1 \le j \le n-1\, .
\]
On the other hand, as  is well known, 
\[ [R:A] \, \mathbb Z= I_2
\left( \begin{matrix}
a_1 & a_2 & \ldots & a_n \\
d-a_1 & d-a_2 & \ldots  & d-a_n
\end{matrix}
\right)\, .
\]
After row and column operations on the matrix, this determinantal ideal becomes 
\[d\cdot I_2\left( \begin{matrix}
0 & a_2-a_1 & \ldots & a_n-a_{n-1} \\
1 & 0 & \ldots  & 0
\end{matrix}
\right)=d \cdot {\rm gcd}(D_1, \ldots, D_{n-1}) \,  \mathbb Z\, .
\]
It follows that \[[B:A]=\frac{[R:A]}{d}={\rm gcd}(D_1, \ldots, D_{n-1})= {\rm gcd}(\mbox{column degrees of \,} \varphi)\, .\]
\end{proof}

\begin{remark} Adopt the data of {\rm \ref{data0}}.  If $I$ is a monomial ideal \[\operatorname {gcd}(\text{column degrees of
$\varphi$})=1\iff A \subset B \text{ is a birational extension}.\]
However, the direction $(\Leftarrow)$ is far from true if
 the ideal $I$ is not monomial; see, for example, Cor.~\ref{corn}(4) or all of \cite{CKPU}. 
\end{remark}


The following theorem summarizes some of the results of  the present section that were used in \cite{CKPU}. It was  recorded  as \cite[Thm.~0.10]{CKPU}. (The present statement is the correct one.)

\begin{theorem}\label{2}  Let  $R$ be the standard graded polynomial ring $k[x,y]$, with $k$ an infinite perfect field,  
 $I$ be a height two  ideal of  $R$ generated by forms $g_1,\dots,g_n$
     of degree $d$, and $\varphi$   be a homogeneous Hilbert-Burch matrix for the row vector  $[g_1,\dots,g_n]$.
       If $A$ and $B$  are  the standard graded $ k$-algebras $A= k[I_d]$ and  
$B=k[R_d]$, $r$ is the degree of the field extension $[\operatorname{Quot}(B):\operatorname{Quot}(A)]$, and $e$ is the multiplicity of  $A$,  then the following statements hold.
\begin{enumerate}[\quad\rm(1)]
\item $re=d$.
\item The morphism $\Bbb P^1_k\to \Bbb P_k^{n-1}$, which is given by $q\mapsto  [g_1(q):\dots:g_n(q)]$, is birational onto its image if and only if $e=d$.
\item There exist 
forms $f_1$ and $f_2$  of degree $r$ in $R$ such that $A\subset k[f_1,f_2]$. 
 In particular, $I$ is extended from an ideal in
$ k[f_1,f_2]$ in the sense that ${I=(I\cap  k[f_1,f_2])R}$.    Furthermore,   we can choose $f_i=\gcd(I_1(p_i\varphi))$ where $p_i=[g_1(q_i):\dots:g_n(q_i)]$
  for general points ${q_1}$ and  ${q_2}$  in
$\Bbb P^1_k$.
\end{enumerate}\end{theorem}
\begin{proof} Item (1) is Observation~\ref{der}; (2) follows from (1) and the definition of birationality; and (3) is Theorem~\ref{repar} and Corollary~\ref{H}.\end{proof}

\section{Rees rings}\label{Reesrings}

\begin{data}\label{d2} Let  $R$ be a standard graded domain over a field $k$, of dimension $s$, with maximal homogeneous ideal $\m$. Let $I$ be a homogeneous  ideal  in $R$ generated by  forms  of degree $d$, $A$ be the ring $k[I_d]$, and $B$ be the Veronese ring $k[R_d]$. After regrading, we view $A\subset B$ as standard graded $k$-algebras. 
\end{data}

In this section, we relate properties of the extension $A\subset B$ to properties  of the corresponding inclusions of Rees algebras $\R(I)\subset \R(\m^d)$. The extension of Rees algebras has the advantage that it is automatically birational.

\begin{remark}\label{R1} Adopt the data of {\rm \ref{d2}}. The injective map of graded $k$-algebras \[A \cong k[I_dt] \hookrightarrow R[It]=\R(I)\] induces an isomorphism between the ring $A$ and the special fiber ring $ \R(I)/\m \R(I)$; see, for example, (\ref{typical}). The ring $A$ is a domain; hence, $\m \R(I)$ is a prime ideal. Since  ${\rm dim } \, \R(I)$ is equal to $s+1$ (see, for example, \cite[5.1.4]{SH}), we obtain 
\[{\rm dim}\, A =  s+1 - {\rm dim} \,  \R(I)_{\m \R(I)}\, .\]
\end{remark}

In the next result we show that the degree of the morphism $\Psi: \Bbb P_k^{s-1} \longrightarrow {\rm Im}\,  \Psi$ is equal to the multiplicity of the local ring $\R(I)_{\m \R(I)}$. In addition, we relate the degree of the morphism $\Psi$, the degree of the image, and the $j$-multiplicity of the ideal $I$. The notion of $j$-multiplicity and the importance of this result is discussed in the introduction and in Corollary \ref{nov24}.

\begin{theorem}\label{rees}
Adopt the data of {\rm \ref{d2}}.  
 If the Krull dimension of $A$ is equal to $s$, then 
\[e(\R(I)_{\m \R(I)})=[B: A] \qquad \mbox{and} \qquad j(I)=d\cdot [B: A] \cdot e(A)\, .\]
\end{theorem}
\begin{proof} Inside the Rees ring $\R:=\R(I)=R[It]\subset R[t]$ we consider the $k$-subalgebra $A'=k[I_dt]$, which is  isomorphic to $A$. Write $K$ for the quotient field of $A'$. We notice that $\R \otimes_{A'} K$ is a standard graded $K$-algebra with maximal homogeneous ideal $\m (\R \otimes_{A'} K)$ and $(\R \otimes_{A'} K)_{\m (\R \otimes_{A'} K)}$ is equal to $ \R_{\ \m \R}$. Therefore the local ring  $\R_{\ \m \R}$ and the standard graded $K$-algebra $\R \otimes_{A'} K$  have the same multiplicity and the same Krull dimension. 
This dimension is one by Remark ~\ref{R1}.

Let $g$ be a non-zero homogeneous element of  $I_d$.  Notice that $g$ is a homogeneous non-zerodivisor  of degree $d$ in the one-dimensional standard graded $K$-algebra $\R\otimes_{A'} K$. We obtain
\begin{eqnarray*} e(\R \otimes_{A'} K)
=e((\R \otimes_{A'} K)^{(d)})&=&\left[\fakeht (\R \otimes_{A'} K)^{(d)}: K[g]\fakeht\right]\\
&=&
{\textstyle\frac{1}{d}}\, \left[(\R \otimes_{A'} K):K[g]\vphantom{E^{E_E}_{E_E}}\right]\\
&=&
{\textstyle\frac{1}{d}}\,  \left[R[t]:A[t]\fakeht\right]\, .
\end{eqnarray*}
The last equality holds because \[{\rm Quot} (K[g])={\rm Quot}(k[I_dt,g])={\rm Quot}(k[I_d,t])={\rm Quot}(A[t])\] and  ${\rm Quot}(\R \otimes_{A'} K)={\rm Quot}(R[t])$. 
Now we conclude that 
\[e(\R_{\ \m \R})=e(\R \otimes_{A'} K)={\textstyle\frac{1}{d}}\, [R:A]=[B:A]\, .
\] 
This concludes the proof of the first equality. 

The second equality follows from \cite[6.1(a) and 3.4]{SUV}. Here we give a self-contained proof. We write $\G:={\rm gr}_I(R)=\R/I\R$ for the associated graded ring of $I$. Notice that $\G/\m\G\cong \R/\m \R \cong A$. In particular, $\m \G$ is a prime ideal in $\G$ and $\G_{\m \G}$ is Artinian, because $A$ is a domain and ${\rm dim} \, A=s= {\rm dim}\, \G$. 

The $j$-multiplicity of $I$ is 
\[j(I)=e(0:_{\G} \m^{\infty})\, .\]
Since ${\rm Supp}_{\G}(0:_{\G} \m^{\infty})=V(\m \G)$, the associativity formula for multiplicity \ref{assoc} gives
\[e(0:_{\G} \m^{\infty})=\lambda_{\G_{\m \G}}((0:_{\G} \m^{\infty})_{\m \G})\cdot e(\G/ \m \G)=\lambda(\G_{\m \G})\cdot e(A)\, .\]
We conclude that 
\[j(I)=\lambda(\G_{\m \G})\cdot e(A)\, .\]
(The same equality was proved in \cite[the proof of 3.1]{JMV}.)

If $g\not= 0$ is an element of $I_d$ as above, then $gt \in \R \setminus \m \R$, hence $\frac{I}{g}=\frac{It}{gt} \subset \R$. It follows that  $I \R_{\ \m\R} = g \R_{\ \m\R}$, which gives
\[\G_{\m\G}\cong \R_{\ \m\G}/(g)\, .\]
Now
\[\lambda(G_{\m\G})=\lambda(\R_{\ \m\G}/(g))=e((\R\otimes_{A'} K)/(g))=d \cdot e(\R\otimes_{A'} K)=d\cdot e(\R_{\ \m \R})= d \cdot [B:A]\, .
\]
We conclude that $j(I)=d\cdot [B: A] \cdot e(A)$. 
\end{proof}
 
\smallskip

In the next corollary we express the degree of certain dual varieties 
in terms of the  $j$-multiplicity of Jacobian ideals.

 \begin{corollary} Let $k$ be an algebraically closed field of characteristic zero, let $X=V(f)\subset \Bbb P_k^{n-1}$ be a reduced and irreducible hypersurface of degree $d>1, $ and assume that the dual variety $X' \subset \Bbb P_k^{{n-1}'}$ is non-deficient, i.e., is a hypersurface as well. Let $R$ denote the homogenous coordinate ring of $X\subset \Bbb P_k^{n-1}$, and $I=(\partial f/\partial x_1, \ldots, \partial f/\partial x_n)R$ its Jacobian dual. Then 
 \[{\rm deg} \, X'=j(I)/(d-1)\]
 \end{corollary}
\begin{proof} We apply the notation of Data \ref{d2} to the present setting. The ideal $I$ is generated by forms of degree $d-1$, and the algebra $A$ is the homogenous coordinate ring of $X'  \subset \Bbb P_k^{{n-1}'}$. According to \cite[Theorem 4]{K} or \cite[Proposition 3.3]{P}, the extension $A\subset B$ is birational, i.e., $[B:A]=1$. Now Theorem \ref{rees} shows that  $j(I)=(d-1){\rm deg} \ X'$. 
 \end{proof}

\smallskip

\begin{corollary}\label{C0} Adopt the data of {\rm \ref{d2}}.   
Let $J$ be a reduction of $I$ generated by $s$ forms of degree $d$. If the Krull dimension of $A$ is equal to $s$, then 
\[e(A)=\frac{e(\R(J)_{\m \R(J)})}{e(\R(I)_{\m \R(I)})}\  .
\]
\end{corollary}
\begin{Remark}If $k$ is infinite, then $I$ has a reduction generated by $s$ forms of degree $d$.
\end{Remark}
\begin{proof} Notice that $k[J_d]$ is a  homogeneous Noether normalization of $A=k[I_d]$. Therefore 
$e(A)=[k[I_d]:k[J_d]]$. On the other hand, \[[k[I_d]:k[J_d]]=\frac{[B:k[J_d]]}{[B:k[I_d]]}=\frac{e(\R(J)_{\m \R(J)})}{e(\R(I)_{\m \R(I)})}\, , \]
where the last equality follows from  Theorem~\ref{rees} applied to the two ideals $J$ and $I$, respectively. 
\end{proof}

As a consequence we obtain the following numerical criterion for birationality in terms of the Rees ring or the special fiber ring, respectively. 

\begin{corollary}\label{C1}Adopt the data of {\rm \ref{d2}}.    The following statements are equivalent.
\begin{enumerate}[\quad\rm(1)]
\item The ring extension $A \subset B$ is birational, i.e. $[B:A]=1$. 
\item The ring $\R(I)_{\m \R(I)}$ is a discrete valuation ring.
\item The ring $A$ has dimension $s$ and $e(A)=e(\R(J)_{\m \R(J)})$, where $J$  is  a reduction of $I$ generated by $s$ forms of degree $d$.\end{enumerate}
\end{corollary} 
\begin{proof} Notice that if 
the ring extension $A \subset B$ is birational, then $A$ and $B$ have the same transcendence degree over $k$ and therefore ${\rm dim} \, A={\rm dim}\, B=s$. Furthermore ${\rm dim}\, \R(I)_{\m\R(I)} = 1$ if and only if ${\rm dim}\, A=s$ according to Remark \ref{R1}. Now the asserted equivalences follow from Theorem~\ref{rees} and Corollary~\ref{C0}.
\end{proof}

More can be said if  the rational map $\Psi$ of  data of {\rm \ref{d1}}  is a morphism, or equivalently, if  the ideal $I$ is $\m$-primary. A crucial statement in the following corollary is the statement about canonical modules. Indeed, cores, adjoints, and  $S_2$-fications are all constructed using the canonical module  of the Rees ring.  

\begin{remark}\label{S_2} If $I$ is an $\m$-primary ideal generated by forms of degree $d$, then the ideal $\m^d$ is integral over $I$. Hence the ring extension \[\R(I) \subset \R(\m^d)\] is not only birational, but also 
module-finite. Thus applying the functor $-^{\vee}={\rm Hom}_{\R(I)}(-, \omega_{\R(I)})$ induces a containment of the canonical modules
\[ \omega_{\R(\m^d)}\subset  \omega_{\R(I)}\, , \]
and evaluating $-^{\vee}$ twice yields an inclusion of the $S_2$-ifications
\[\End_{\R(I)}(\omega_{\R(I)})\subset  \End_{\R(\m^d)}(\omega_{\R(\m^d)})\, .
\]
Applying $-^{\vee}$ once more to last inclusion returns the previous one. In particular, one of the two last inclusions is an equality if and only if  the other is.  
\end{remark}

\begin{corollary}\label{C2}Adopt the data of {\rm \ref{d1}}.  In addition assume that $I$ is $\m$-primary and the  field $k$ is infinite. 
Statements {\rm(1)\,--\,(5)} are equivalent.
\begin{enumerate}
\item[{\rm (1)}]  The morphism  $\, \Psi$ is birational onto its image, i.e. $[B:A]=1$. 
\item[{\rm (2)}]  The Rees ring $\R(I)$ satisfies Serre's condition $(R_1)$. 
\item[{\rm (3)}]  $\omega_{\R(I)}= \omega_{\R(\m^d)}$.
\item[{\rm (4)}]   $\End_{\R(I)}(\omega_{\R(I)})=  \End_{\R(\m^d)}(\omega_{\R(\m^d)})$.
\item[{\rm (5)}]   $e(A)=d^{s-1}$.\end{enumerate}
Furthermore,  statements {\rm(1)\,--\,(5)}  all  imply
\begin{enumerate}
\item[{\rm (6)}]  
${\rm gradedcore}(I)={\rm core}(I)={\rm core}(\m^d)=\m^{(d-1)s+1}={\rm
 adj}(I^{s})$.
 \item[{\rm (7)}]  If $\R(I)$ satisfies Serre's property $(S_2)$, then $\R(I)$ is Cohen-Macaulay.
\end{enumerate}
\end{corollary}
\begin{proof}
Since $I$ is $\m$-primary the ring $A$ has Krull dimension $s$. In particular, $\m \R(I)$ is a prime ideal of  height one. If $Q$ is any other  height one prime ideal of $\R(I)$, then $Q$ does not contain $\m$ and hence does not contain $I$. It follows that $\R(I)_Q$ is regular. Thus $\R(I)$ satisfies Serre's condition $(R_1)$ if and only if $\R(I)_{\m \R(I)}$ is a discrete valuation ring. Now the equivalence of (1) and (2) follows from Corollary~\ref{C1}. 

To see the equivalence of (2) and (3) notice that $\R(\m^d)$ is the integral closure of $\R(I)$. Hence $\R(I)$ satisfies Serre's condition $(R_1)$ if and only if the $\R(I)$-module $\R(\m^d)/\R(I)$ has codimension at least two, which in turn is equivalent to the equality $\omega_{\R(I)}= \omega_{\R(\m^d)}$. The second equivalence is obtained by analyzing the long exact sequence of ${\rm Ext}^{\bullet}_{R\otimes_k S}(-, \omega_{R\otimes_k S})$ associated to the short exact sequence
\[ 0 \longrightarrow \R(I) \longrightarrow \R(\m^d) \longrightarrow \R(\m^d)/\R(I)\longrightarrow 0
\]
and bearing in mind that $\R(\m^d)$ is Cohen-Macaulay. 
The equivalence of (3) and (4) is explained in Remark~\ref{S_2}.

To see that (5) is equivalent to the other statements we apply Corollary~\ref{C1}. We may assume that $k$ is infinite. It suffices to show the equality  $e(\R(J)_{\m \R(J)})=d^{s-1}$ for the ideal $J$ of Corollary~\ref{C1}. Since $I$ is $\m$-primary, the ideal $J$ is generated by a  regular sequence of $s$ forms of degree $d$. Therefore, $[R: k[J_d]]=d^s$ and then $e(\R(J)_{\m \R(J)})= [B:k[J_d]]=d^{s-1}$ according to Theorem~\ref{rees} .

We  now show that (2) implies (6). We will repeatedly use the fact from \cite[2.1]{PUV} that the core of $\m$-primary ideals localizes. The first equality follows from  \cite[4.5]{CPU1} as $I$ is generated by forms of the same degree.  Recall that the Rees ring $R[It]$ and the extended Rees ring $R[It,t^{-1}]$ have isomorphic projective spectra. Hence if the Rees ring satisfies $R_1$, then  so does $\Proj(R[It,t^{-1}])$. It follows that $R[It,t^{-1}]$ satisfies $R_1$ as well because it suffices to consider prime ideals containing $t^{-1}$ and because the ideal $( It, t^{-1})R[It,t^{-1}]$ has height $s+1>1$. One uses the argument that (2) implies (3) to see that the $R_1$ property of the extended Rees ring implies the bi-graded isomorphism \[\omega_{R[It, t^{-1}] }\cong \omega_{R[\m^dt, t^{-1}]}\, .\]  
By \cite[2.2]{PU} and \cite[1.2]{FPU} the core can be recovered from the canonical module of the extended Rees ring. Therefore $\core(I)=\core(\m^d)$, which is the second equality in (6). The third equality follows from \cite[4.2]{CPU2}. Finally, one has $\m^{(d-1)s+1}={\rm adj} (\m^{ds})$ by \cite[1.3.2(c)]{L} and ${\rm adj}(\m^{ds})={\rm adj}\, (I^{s})$ because $\m^{ds}$ is integral over $I^s$. 

To see (7), notice that $\R(I)$ is normal, hence $\R(I)=\R(\m^d)$. The last ring is Cohen-Macaulay because it is a direct summand of 
$\R(\m)$, which is a Cohen-Macaulay ring and a finitely generated module over $\R(\m^d)$.
%
\end{proof}

The canonical module  and the $S_2$-ification of $\R(m^d)$ that appear in Corollary~\ref{C2} are known explicitly.  


\begin{remark}\label{canonical} Let  $R=k[x_1,\dots,x_s]$ be a standard graded polynomial ring  in $s$ variables over a field $k$, with maximal homogeneous ideal $\m$. Let $d$ be a positive integer and write $s-1=qd+r$ where $q$ and $r$ are non negative integers with $r<d$. The bi-graded canonical module of the Rees ring of  $\m^d$ is \[\omega_{\R(\m^d)}=\begin{cases}x_1^s t \, \m^{d-s+1}\R(\m^d) & \qquad  \mbox{if} \qquad s \le d\\
x_1^st((1,t)^{q-1}, \m^{d-r}t^{q})\R(\m^d) & \qquad  \mbox{if} \qquad d<s\, .\\
\end{cases}\]
Furthermore, the natural inclusion
$\R(m^d) \subset \End_{\R(m^d)}(\omega_{\R(m^d)})$
is an equality.
\end{remark}
\begin{proof} From the presentation of the Rees ring \[\R(\m)= \frac{R[T_1, \ldots, T_s]}{I_2\left(\bmatrix x_1 & \ldots & x_s \\ T_1 & \ldots & T_s\endbmatrix\right)}\] one computes  \[\omega_{\R(\m)}\cong x_1^2t (x_1,x_1t)^{s-2} \R(\m)\, ,\]
where the factor $x_1^2t$ is required to make the isomorphism bi-homogeneous. The Rees ring $\R(\m^d)$ is obtained by taking the $d$-Veronese sub-ring of $\R(\m)$ with respect to the grading given by the $T$-variables and then rescaling the grading. Since the Veronese functor commutes with taking canonical modules, we obtain the desired statement.

The natural inclusion
$\R(m^d) \subset \End_{\R(m^d)}(\omega_{\R(m^d)})$
is an equality because $\R(m^d)$ is Cohen-Macaulay and $\End_{\R(m^d)}(\omega_{\R(m^d)})$ is its $S_2$-ification. 
\end{proof}

In Corollaries~\ref{R-sub-i} and \ref{iso-singu} we extend the equivalence $(1)\iff (2)$ from Corollary~\ref{C2}. We begin by comparing the heights of conductors  for certain   extensions of Rees rings and special fiber rings.

\begin{proposition}\label{conductors}Adopt the data of {\rm\ref{d2}}. In addition, assume that $I$ is $\mathfrak m$-primary. Let $K$ be an ideal of $R$ which is generated by forms of degree $d$ and which contains $I$. Write $C$ for $k[K_d]$; so, $A=k[I_d]\subset C=k[K_d]$. Then 
$$\operatorname{ht} (\mathcal R(I):_{\mathcal R(I)}\mathcal R(K))= 1 + \operatorname{ht}(A:_AC)=1+\operatorname{ht}(\mathcal F(I):_{\mathcal F(I)}\mathcal F(K)).$$
\end{proposition}

\begin{proof} Let $A'$  be the sub-ring $k[I_dt]$  of $\mathcal R(I)=R[I_dt]$ and $C'$ be the sub-ring $k[K_dt]$  of $\mathcal R(K)=R[K_dt]$. 
%
%
We first prove that the ideals 
\begin{equation}\label{eq-rad}\operatorname{rad} (\mathcal R(I):_{\mathcal R(I)}\mathcal R(K)) \quad\text{and}\quad 
\operatorname{rad} ((\mathfrak m, A':_{A'}C')\mathcal R(I))\end{equation}
of $\mathcal R(I)$ are equal. 

We prove the inclusion ``$\supset$'' for the ideals of (\ref{eq-rad}). The Rees algebra $\mathcal R(K)$ is a 
module-finite extension of $\mathcal R(I)$ because $I$ is $\mathfrak m$-primary; hence, $I$ is a reduction of $K$ and  there is a positive  integer $t$ with $K^tK^n\subset I^n$ for all non-negative integers $n$. It follows that 
$K^t\subset \mathcal R(I):_{\mathcal R(I)}\mathcal R(K)$; and therefore, $\mathfrak m\subset \operatorname{rad} (\mathcal R(I):_{\mathcal R(I)}\mathcal R(K))$. Clearly, $A':_{A'}C'\subset \mathcal R(I):_{\R(I)}\R(K)$ because $\R(I)=R[A']$ and $\R(K)=R[C']$.

We prove the inclusion ``$\subset$'' for the ideals of (\ref{eq-rad}). Let $\alpha$ be a bi-homogeneous element of $\mathcal R(I):_{\mathcal R(I)}\mathcal R(K)$. In particular,   $\alpha=ft^j$ for some homogeneous form $f$ in $I$ and some non-negative integer $j$. If the homogeneous form $f$ has degree $i+dj$ as a polynomial in 
$R=k[x_1,\dots,x_s]$, then the bi-degree of $\alpha$ in $\mathcal R(I)$ is $(i,j)$. If $i$ is positive, then $\alpha$ is in $\mathfrak m\mathcal R(I)$. If $i=0$, then $\alpha$ is in $A'$ and $\alpha\cdot C'\subset \alpha \cdot\mathcal R(K)\subset \mathcal R(I)$. One may compute $x$-degree to see that $\alpha \cdot C'$ is actually contained in $A'$. In either case, $\alpha\in  (\mathfrak m, A':_{A'}C')\mathcal R(I)$. 

Now that the assertion of (\ref{eq-rad}) has been established, we complete the proof of the result by applying the natural quotient  homomorphism $\mathcal R(I)\to \mathcal R(I)/\mathfrak m\mathcal R(I)$
to the ideals of (\ref{eq-rad}). The 
ring $\R(I)$ is a domain and a finitely generated algebra over a field, and the 
ideal $\mathfrak m\mathcal R(I)$ of $\mathcal R(I)$ 
has height $1$. 
Thus,
$$\begin{array}{lllll}\operatorname{ht} (\mathcal R(I):_{\mathcal R(I)}\mathcal R(K))-1&{}={}& \operatorname{ht} \frac{\operatorname{rad} (\mathcal R(I):_{\mathcal R(I)}\mathcal R(K))}{\mathfrak m\mathcal R(I)}
&{}={}&\operatorname{ht} (\mathfrak m, A':_{A'}C')\frac{\mathcal R(I)}{\mathfrak m\mathcal R(I)}\vspace{5pt}\\
&{}={}&\operatorname{ht} (\mathcal F(I):_{\mathcal F(I)}\mathcal F(K))&{}={}&\operatorname{ht} (A:_AC)
.\end{array}$$
\end{proof}

\begin{corollary}\label{R-sub-i}Adopt the data of {\rm\ref{d2}}. In addition, assume that $R=k[x_1, \ldots, x_s]$ and $I$ is $\mathfrak m$-primary. Let $i\le s$ be a fixed positive integer. Then, the following conditions are equivalent.
\begin{enumerate}[\rm(a)]
\item $\mathcal R(I)$ satisfies the Serre condition $(R_i)$. 
\item $A$ satisfies the Serre condition $(R_{i-1})$ and $A\subset B$ is a birational extension.
\end{enumerate}
\end{corollary}
\begin{proof} When $i=1$, the statement is $(1)\iff(2)$ from Corollary~\ref{C2}. Henceforth, we assume $2\le i$ and that $A\subset B$ is a birational extension. Thus, both extensions $\R(I)\subset \R(\m^d)$ and $A\subset B$ are module-finite and birational.

Moreover, $\dim \R(I)=s+1$ and $\R(\m^d)$ has an isolated singularity. Hence for any prime ideal $Q$ of $\R(I)$ with $\dim\R(I)_Q\le s$, the ring  $\R(I)_Q$ is regular if and only if $\R(I)_Q= \R(\m^d)_Q$. This shows that $\R(I)$ satisfies  $(R_i)$ if and only if $i<\operatorname{ht}(\R(I):_{\R(I)}\R(\m^d))$.

Similarly, $\dim A=s$ and $B$ has an isolated singularity. Hence again, $A_q$ is regular if and only if $A_q=B_q$ for any prime ideal $q$ of $A$ with $\dim A_q\le s-1$. We conclude that $A$ satisfies $(R_{i-1})$ if and only if $i-1<\operatorname{ht}(A:_AB)$.

Apply Proposition \ref{conductors}, with $K=\m^d$ to complete the proof. \end{proof}

In the situation of Data~\ref{d1},  $A$ is the homogeneous coordinate ring of the image of $\Psi$. The following geometric consequence of Corollary~\ref{R-sub-i} is now immediate.
\begin{corollary}\label{iso-singu}Adopt the data of {\rm \ref{d1}}.  In addition assume that $I$ is $\m$-primary.
Then the Rees ring $\R(I)$ has an isolated singularity if and only if 
the image of \,$\Psi$ is smooth  and $\Psi$ is birational onto its image.
\end{corollary}

For $s=2$ we harvest consequences about canonical modules, $S_2$-fications, and  cores for any ideal generated by forms of the same degree. Indeed, when $s=2$ we can pass by way of a faithfully flat descent to the birational situation according to Corollary~\ref{H}. We compute canonical modules, $S_2$-fications, and  cores in the birational situation using Corollary~\ref{C2} and Remark~\ref{canonical}. We can pull  these objects back  because  they are preserved under faithfully flat extensions, whereas adjoint ideals are not preserved under such extensions. 

\begin{theorem}\label{core}  Adopt the data of {\rm \ref{d2}}. In addition assume  that $R=k[x_1, x_2]$ is a polynomial ring in two variables over an infinite field and $I$ is $\m$-primary. Write  $r=[B:A]$ and let 
$f_1$ and $f_2$ be the forms in $R$ of degree $r$ of  Corollary~\ref{H}.  The following statements hold.\begin{enumerate}[\quad\rm(1)]
\item $\omega_{\R(I)}=\omega_{\R((f_1,f_2)^{d/r})}=f_1^2t(f_1,f_2)^{\frac{d}{r}-1}\R((f_1,f_2)^{d/r})$.
\item The $S_2$-fication of the Rees ring $\R(I)$  is  \[{\rm End}\, (\omega_{\R(I)})=\R((f_1,f_2)^{d/r})\, , \]which is a Cohen-Macaulay ring. 
\item $e(A)=\frac{d}{r}$.
\item ${\rm gradedcore}(I)={\rm core}(I)=(f_1,f_2)^{2\frac{d}{r}-1}$.
\item If $\R(I)$ satisfies Serre's property $(S_2)$, then $\R(I)$ is Cohen-Macaulay.
\end{enumerate}
\end{theorem}
\begin{proof} We write $R'=k[f_1,f_2]$ and $I'=I\cap R'$. By Corollary~\ref{H} the ideal $I$ is extended from $I'$, i.e. $I=I'R$. We think of $R'$ as a standard graded polynomial ring in the variables $f_1$ and $f_2$ with maximal homogeneous ideal $\m'$. The ideal $I'$ is generated by forms of degree $d/r$ in this ring. From Corollary~\ref{H} we know that the extension $k[I'_{d/r}]\subset k[R'_{d/r}]$ is birational . Now Corollary~\ref{C2} 
and Remark~\ref{canonical} imply the asserted statements for the ideals $I'\subset (\m')^{d/r}=(f_1,f_2)^{d/r}R'$ in $R'$. To pass back to the ideals $I=I'R \subset (f_1,f_2)^{d/r}R$ in $R$ we use the fact that the map $R'\subset R$ is  flat with Gorenstein fibers. We remind the reader that 
\[\core(I)=\core(I'R)=\core(I')R\, ,\]
according to \cite[2.1]{PUV}  and \cite[4.8]{CPU1}. 
\end{proof}

Remarkably, if $s=2$ then all the statements of Corollary~\ref{C2} are equivalent. More specifically, the  integral  closedness of the core, which is a single graded component of the canonical module of the Rees ring, forces the shape of the entire canonical module. 
 
\begin{corollary}\label{C3}Adopt the data of {\rm \ref{d1}}.   In addition assume that $I$ is $\m$-primary, $s=2$, and the field $k$ is infinite.
%
Statements {\rm(1)\,--\,(8)} are equivalent.
\begin{itemize}
\item[\rm(1)] The morphism  $\, \Psi$ is birational onto its image, i.e. $[B:A]=1$. 
\item[\rm(2)] The Rees ring $\R(I)$ satisfies Serre's condition $(R_1)$. 
\item[\rm(3)] $\omega_{\R(I)}= \omega_{\R(\m^d)}$.
\item[\rm(4)]  $\End_{\R(I)}(\omega_{\R(I)})=  \End_{\R(\m^d)}(\omega_{\R(\m^d)})$.
\item[\rm(5)] $e(A)=d$.
\item[\rm(6)] ${\rm core}(I)=\m^{2d-1}$.
\item[\rm(7)] ${\rm core}(I)={\rm
 adj}(I^{2})$.
 \item[\rm(8)] The ideal ${\rm core}(I)$ is integrally closed.
\end{itemize}
  Furthermore, statements {\rm(1)\,--\,(8)} are all  implied by
\begin{enumerate}
\item[{\rm (9)}]  $\gcd($column degrees of $\varphi)=1$.
\end{enumerate}
\end{corollary}
\begin{proof} Corollary~\ref{C2} shows that items (1)~--~(5) are equivalent and that they imply (6) and (7). Item (8) follows immediately from (6) or (7). We prove that (8) implies (1). From Theorem~\ref{core}, we have ${\rm core}(I)=(f_1,f_2)^{2\frac{d}{r}-1}$. Hence  $\core(I)$ is an $\m$-primary ideal generated by $2\frac{d}{r}$ forms of degree $2d-r$. If such an ideal is integrally closed, it would have to be $\m^{2d-r}$, which is minimally generated by $2d-r+1$ forms. This forces $r=1$.  Finally we appeal to Corollary~\ref{corn}(3) to see that (9) implies (1). 
\end{proof}

\begin{remarks}\begin{enumerate}[\rm(a)] \item If $I$ is generated by monomials then statements {\rm(1)\,--\,(9)} of Corollary ~\ref{C3} are all equivalent according to Observation~\ref{min-ideal-col-deg}.\item When  statements {\rm(1)\,--\,(8)} of Corollary ~\ref{C3} hold, then it follows from
 Corollary~\ref{C2}  that all of the equalities
$${\rm gradedcore}(I)={\rm core}(I)={\rm core}(\m^d)=\m^{2d-1}={\rm
 adj}(I^{2})$$
hold.
\item  Angela Kohlhaas \cite{K10} has proven that if  $I \subset k[x_1, \ldots, x_s]$ is an $\m$-primary  monomial ideal which has a 
reduction generated by $s$ monomials, then the following three conditions are equivalent.
 \begin{itemize}
\item[(2)]  $R[It]$ satisfies Serre's condition $R_1$.
\item[(7)] $\core(I)={\rm adj}(I^s)$.
\item[(8)]   $\core(I)$ is integrally
closed.
 \end{itemize}\end{enumerate}
\end{remarks}

\bigskip\noindent{\bf Acknolwedgment.} This paper was written at
 Centre International de Rencontres Math\'ematiques (CIRM) in Luminy, France, while the authors participated in a Petit Groupe de Travail. The authors are very appreciative of the hospitality offered by the Soci\'et\'e Math\'ematique de France.

\end{document}